\documentclass[a4paper, 12pt]{amsart}
\pdfoutput=1
\usepackage{amsmath, amssymb, amsthm, amsfonts}
\usepackage[mathscr]{eucal}
\usepackage{verbatim}
\usepackage{enumerate}
\usepackage{color}

\newcommand{\C}{\mathbb{C}}
\newcommand{\R}{\mathbb{R}}

\newcommand{\Hilb}{\mathcal{H}}

\newtheorem{lemma}{Lemma}[section]
\newtheorem{theorem}[lemma]{THEOREM}
\newtheorem{definition}[lemma]{Definition}

\title{Admissibility for $\alpha$-Modulation Spaces}
\author{Peter Balazs, Dominik Bayer and Michael Speckbacher}
\thanks{Acoustics Research Institute, Austrian Academy of Sciences, Wohllebengasse 12-14, 1040 Wien, Austria}

\begin{document}

\begin{abstract}

This paper is concerned with frame decompositions of $\alpha$-modulation spaces. These spaces can be obtained as coorbit spaces for square-integrable representations of the affine Weyl-Heisenberg group modulo suitable subgroups. The theory yields canonical constructions for Banach frames or atomic decompositions in $\alpha$-modulation spaces. A necessary ingredient in this abstract machinery is the existence of generating functions that are admissible for the representation.

For numerical purposes, admissible atoms with compact support are necessary. We show new admissibility conditions that considerably improve upon known results. In particular, we prove the existence of admissible vectors that have compact support in time domain.

\end{abstract}

\maketitle


\section{Introduction}

To know where a sound comes from can be essential for humans, in particular in traffic situation. We use the particular filtering effect of our ears, the head and torso to decide, where a signal comes from. These so-called head related transfer functions (HRTFs) \cite{Moelleretal95} are different for each person. The individual difference becomes important for up/down and front/back decisions, necessitating measuring or estimating HRTFs for each subject. HRTF measurements can be quite time-consuming (for an improved method see e.g. \cite{Majdaketal07}). Another way to obtain them is to use a numerical simulation, using the boundary element method (BEM) on a 3D model of the head \cite{Kreuzeretal09b}. The boundary integral operators for the Helmholtz equation have to be discretized, reduced to finite dimension and the integrals to be computed by quadrature to assemble the stiffness matrix need to be solved on a domain with finite support (in general, triangular or quadrilateral patches). It is part of the multi-national, multi-disciplinary project BIOTOP to address the question how frames can be used to design efficient numerical methods to solve the boundary integral equation for the Helmholtz operator.

For the discretization respectively representation of operators (and functions) orthonormal or Riesz bases are often used. An alternative approach is to use frames, which give non-unique coefficients, but still allow a representation of functions \cite{ole1} and operators \cite{xxlframoper1}. Due to their high redundancy and thus flexibility, these systems have a higher probability to produce sparse representations. Furthermore, redundant systems are robust with respect to loss of coefficients.

In acoustics, harmonic components which are well represented by Gabor frames \cite{feistro1} as well as impact sounds which are well represented by wavelet frames \cite{daubech1}, are important. For example, when dealing with acoustic scattering, the solution of the scattering problem can contain harmonic as well as non-smooth components, depending on the geometry of the scatterer. So it seems natural to use a representation, which is in some sense intermediate between the Gabor and wavelet setting, combining the strengths of both. This can be achieved using $\alpha$-modulation frames.

These are frames associated with representations of the affine Weyl-Heisenberg group. Since there do in fact not exist any square-integrable representations of this group, the theory has to be generalized further to include square-integrable representations modulo quotients \cite{DahForRau+2008}.

The so-called $\alpha$-modulation spaces are particular function spaces that can be obtained as coorbit spaces associated to those representations. They were originally introduced and defined by Gr{\"o}bner in 1992 using Fourier domain decomposition methods \cite{Gr_92,FeiGr_85}. The $\alpha$-modulation spaces depend on a parameter $\alpha \in [0,1]$. For $\alpha = 0$, $\alpha$-modulation spaces coincide precisely with the so-called modulation spaces that are defined in terms of Gabor analysis; for $\alpha = 1$, the spaces become Besov spaces, the smoothness spaces associated to the continuous wavelet transform.  It is to be expected that $\alpha$-modulation frames for some intermediate $\alpha$ between $0$ and $1$ might be able to preserve and combine the good properties of both paradigms into one single framework. This idea has already been used successfully in applications, see e.g. \cite{DahTes08}. 

For numerical applications, finite domains (e.g. the surface of a scatterer) are necessary, so the assumption of a band-limited generating function as in \cite{DahForRau+2008} to show the admissibility cannot be taken. In this paper we lift this assumption, substituting it by a weaker condition, namely just a certain decay at infinity of the Fourier transform. 

The paper is organized as follows. In Section \ref{sec:prel0} we review basic facts about generalized representation theory modulo subgroups, coorbit theory and $\alpha$-modulation spaces. In Section \ref{sec:admis0}, we state our main result on the admissibility for possibly non-band-limited generating functions. In particular, we prove the existence of admissible atoms that are compactly supported in time.

\section{Preliminaries} \label{sec:prel0}


\subsection{Representation Theory Modulo Subgroups}

We first give a short outline of the theory of square-integrable group representations, see e.g. \cite{fo95} or \cite{wo02}.
Let $G$ be a locally compact Hausdorff topological group. It is well known that for such groups there always exists a nonzero Radon measure $\mu$, unique up to a constant factor, that is invariant under left translation. This measure is the so-called \emph{(left) Haar measure} of $G$. If the left Haar measure is simultaneously a right Haar measure as well (i.e. it is invariant under right translations), we call the group \emph{unimodular}. Let $\Hilb$ be a separable complex Hilbert space with inner product $\langle\cdot, \cdot \rangle$ and norm $\| \cdot \|$. Denote by $\mathcal{U}(\Hilb)$ the group of unitary operators on $\Hilb$. A \emph{unitary representation} of $G$ on $\Hilb$ is a strongly continuous group homomorphism $\pi:G \to \mathcal{U}(\Hilb)$, i.e. a mapping $\pi:G \to \mathcal{U}(\Hilb)$ such that
\begin{enumerate}
\item $\pi(gh) = \pi(g)\pi(h)$ for all $g,h \in G$, and
\item for every $\phi \in \Hilb$, the mapping $G \to \Hilb$, $g \mapsto \pi(g)\phi$ is continuous.
\end{enumerate}
The group representation is \emph{irreducible} if it has only the trivial invariant subspaces $\{0\}$ and $\Hilb$, i.e. the only closed subspaces $M \subseteq \Hilb$ such that $\pi(g)(M) \subseteq M$ for every $g\in G$ are $M=\{0\}$ or $M = \Hilb$.
The group representation is said to be \emph{square-integrable} if it is irreducible and there exists a vector $\phi \not= 0$ in $\Hilb$ such that 
\[
c_\phi := \int_G \left| \langle \phi, \pi(g)\phi \rangle \right|^2\,d\mu(g) < \infty.
\]
Such a vector $\phi$ with $\| \phi \| = 1$ is called an \emph{admissible wavelet}. Its associated \emph{wavelet constant} is $c_\phi$. For an admissible wavelet $\phi$ and $f \in \Hilb$, the \emph{voice transform} or \emph{generalized wavelet transform} $V_{\phi}f:G \to \C$ is defined as
\[
V_{\phi} f(g) := \langle f, \pi(g)\phi \rangle,
\]
for $g \in G$. By square-integrability, we have $V_{\phi}f \in L^2(G, \mu)$, thus $V_{\phi}: \Hilb \to L^2(G,\mu)$. The adjoint of this mapping is
\[
V^\ast_\phi: L^2(G,\mu) \to H, F \in L^2(G,\mu) \mapsto V^\ast_\phi F = \int_G F(g) \pi(g)\phi\,d\mu(g) \in \Hilb,
\]
to be understood in weak sense as
\[
\langle V^\ast_\phi F, h \rangle = \int_G F(g) \langle \pi(g)\phi, h \rangle \,d\mu(g)
\]
for all $h \in \Hilb$.
If $\pi:G \to \mathcal{U}(\Hilb)$ is irreducible and square-integrable, and $\phi\in\Hilb$ is admissible, then we have the following \emph{resolution of the identity}: for all $f \in \Hilb$,
\[
\frac{1}{c_\phi} V_\phi^\ast ( V_{\phi}f ) = \frac{1}{c_\phi} \int_G \langle f, \pi(g)\phi \rangle \pi(g)\phi\,d\mu(g) = f,
\]
that means the \emph{reproducing formula}
\[
\frac{1}{c_\phi} \int_G \langle f, \pi(g)\phi \rangle \langle \pi(g)\phi, h \rangle\,d\mu(g) = \langle f, h \rangle
\]
holds for all $f,h \in \Hilb$. One may interpret the family $\{ \pi(g)\phi\,:\,g \in G\}$ as a \emph{continuous frame}, with $V_\phi$ the \emph{analysis operator}, $V^\ast_\phi$ the \emph{synthesis operator} and $A_{\phi} := V^{\ast}_\phi( V_\phi)$ the \emph{frame operator}.

From here, one proceeds to build classical coorbit theory, as explained in \cite{FeiGro_89, FeiGro_89-2}.

In many cases, however, representations of a group are not square-integrable. The usual informal  interpretation of this fact is that the group is, in a certain sense, too large. Following
\cite{DahForRau+2008}, the subsequent technique may be used to make the group smaller: choose a suitable closed subgroup $H$ and factor out, forming the quotient $G/H$. In general, $H$ need not be a normal subgroup, so that $G/H$ will, in general, not carry a group structure; it is a \emph{homogeneous space}, though, i.e. the group $G$ acts on $G/H$ continuously and transitively by left translation. The quotient can always be equipped, in a natural way, with a measure $\mu$ that is \emph{quasi-invariant} under left translations, i.e. $\mu$ and all its left-translates have the same null sets. In many examples the measure $\mu$ will be translation-invariant in the first place. In order to transfer the representation from the group to the quotient, one then introduces a measurable \emph{section} $\sigma: G/H \to G$ which assigns to each coset a point lying in it. We can then generalize admissibility and square-integrability for representations modulo subgroups in the following definition. 

\begin{definition}\label{D:admissible_abstract}
Let $G$ be a locally compact group, $\pi:G \to \mathcal{U}(\mathcal{H})$ a unitary representation, $H$ a closed subgroup of $G$, and $X = G/H$, equipped with a (quasi-)invariant measure $\mu$. Let $\sigma:X \to G$ be a section and $\psi \in \mathcal{H} \setminus \{0\}$. Define the operator $A_{\sigma, \psi}$ on $\mathcal{H}$ (weakly) by
\[
A_{\sigma, \psi}f := \int_X \langle f, \pi(\sigma(x))\psi \rangle \pi(\sigma(x))\psi\,d\mu(x), \quad f \in \mathcal{H}.
\]
If $A_{\sigma, \psi}$ is bounded and boundedly invertible, then $\psi$ is called \textbf{admissible}, and the unitary representation $\pi$ is called \textbf{square-integrable modulo $(H, \sigma)$}.
\end{definition}

There is also a generalization of the voice transform in this setting.

\begin{definition}\label{D:voice_abstract}
Let $\psi$ be admissible. Then the \textbf{voice transform} of $f \in \mathcal{H}$ is defined by
\[
V_{\psi}\,f(x) := \langle f, \pi(\sigma(x))\psi \rangle, \quad x \in X.
\]
We further define a \textbf{second transform}
\[
W_{\psi}\,f(x) := V_{\psi}(A_{\sigma, \psi}^{-1}\,f)(x) = \langle f, A_{\sigma, \psi}^{-1}\,\pi(\sigma(x))\psi \rangle, \quad x \in X.
\]
\end{definition}

We have then the following version of the reproducing formula (see formula (2.4) in \cite{DahForRau+2008}).

\begin{theorem}
Let $\psi$ be admissible for the representation $\pi$ modulo $(H,\sigma)$. Then, for all $f, g \in \Hilb$,
\[
\langle f, g \rangle = \langle W_\psi f, V_\psi g \rangle_{L^2(G,\mu)} = \langle V_\psi f, W_\psi g \rangle_{L^2(G,\mu)}.
\]
\end{theorem}

The construction of coorbit spaces for the generalized setting proceeds from here analogously to the classical case, see \cite{DahForRau+2008}.


\subsection{The Setting for $\alpha$-Modulation Spaces}

It turns out that $\alpha$-modulation spaces can be constructed as coorbit spaces in the setting of generalized representation theory modulo subgroups for a particular group, the \emph{affine Weyl-Heisenberg group}.

We make precise the abstract theory of the preceding subsection for that case. Again, we closely follow \cite{DahForRau+2008}.

Denote by $\R_+$ the set of positive real numbers. The \emph{affine Weyl-Heisenberg group} is the set
\[ G_{aWH} := \R^2 \times \R_+ \times \R
\]
together with composition law
\begin{align*}
(x,\omega, a, \tau) & \circ (x^{\prime}, \omega^{\prime}, a^{\prime}, \tau^{\prime}) = ( x + ax^{\prime}, \omega + \frac{1}{a}\omega^{\prime}, aa^{\prime}, \tau + \tau^{\prime} + \omega a x^{\prime} ).
\end{align*}
Equipped with the usual product topology of the respective Euclidean topologies on $\R$ and $\R_+$, this becomes a (non-abelian) locally compact Hausdorff topological group. The neutral element is $(0,0,1,0)$, and the inverse to $(x,\omega, a, \tau)$ is $(x, \omega, a, \tau)^{-1} = (-\frac{x}{a}, -\omega a, \frac{1}{a}, - \tau + x\omega)$. The Haar measure is given by
\[ d\mu(x,\omega, a,\tau) = dxd\omega\frac{da}{a}d\tau.
\]
This is in fact both a left and right Haar measure on $G_{aWH}$, thus the group is \emph{unimodular}.

We define the "usual three" operators, \emph{translation}: $T_xf(t) = f(t-x)$, \emph{modulation}: $M_{\omega}f(t) = e^{2\pi i \omega t}f(t)$, and \emph{dilation}: $D_af(t) = \frac{1}{\sqrt{a}}f(\frac{t}{a})$ (with $x, \omega \in \R$ and $a \in \R_+$). These are unitary operators on $L^2(\R)$. Using them, we define the \emph{Stone-Von Neumann representation}, given by
\[
\pi:G_{aWH} \to \mathcal{U}(L^2(\R)), \quad\quad \pi(x, \omega, a, \tau) = e^{2\pi i \tau} T_x M_{\omega} D_a.
\]
This constitutes a unitary representation of $G_{aWH}$, but unfortunately not a square-integrable one.

The subset
\[
H = \{ (0,0,a,\tau) \} \subseteq G_{aWH}.
\]
is a closed subgroup of the affine Weyl-Heisenberg group, although not a normal subgroup. Define the quotient
\[
X := G_{aWH} / H \simeq \R^2.
\]
This is not a group but a \emph{homogeneous space}. It carries the measure $dxd\omega$ which is in fact a truly \emph{invariant} measure under left translations on $X$.

For $0 \leq \alpha < 1$, choose the \emph{section}
\[
\sigma:X \to G_{aWH}, \quad \quad \sigma(x,\omega) = (x,\omega, \beta(\omega), 0)
\]
with
\[
\beta(\omega) = (1 + |\omega |)^{-\alpha}.
\]

One can then show that, for $\psi \in L^2(\R)$, the operator from Definition \ref{D:admissible_abstract} is in this case a Fourier multiplier, in general unbounded, but densely defined:
\[
\widehat{A_{\sigma, \psi} f} = m \cdot \hat{f}
\]
for $f \in \mbox{dom}(A_{\sigma, \psi}) \subseteq L^2(\R)$ a dense subspace, with \emph{symbol}
\[
m(\xi) = \int_{\R} | \hat{\psi}(\beta(\omega)(\xi - \omega)) |^2 \beta(\omega)\,d\omega.
\]
Thus the admissibility of $\psi \in L^2(\R)$ is equivalent to boundedness and invertibility of the Fourier multiplier $A_{\sigma, \psi}$, which is in turn equivalent to the existence of constants $A, B$ such that
\[ 
0 < A \leq m(\xi) \leq B < \infty \tag{$\ast$}
\]
for almost all $\xi \in \R$.

\section{Admissibility} \label{sec:admis0}



We prove a new admissibility condition for $\alpha$-modulation spaces. This generalizes results previously obtained by Dahlke et al. More precisely, whereas in \cite{DahForRau+2008} it was shown that band-limited functions, that is functions with compactly supported Fourier transform, are admissible, we prove that it suffices to demand just a certain decay of the Fourier transform. In particular, we find admissible functions that are compactly supported in time.

In the following, we always set
\[
\beta(\omega) = (1 + |\omega|)^{-\alpha},
\]
with fixed $\alpha \in [0,1)$. Two simple properties of $\beta$, that we will often use without further comment in the sequel, are
\begin{enumerate}[(i)]
\item $\beta$ is symmetric, i.e. $\beta(\omega) = \beta(-\omega)$ for all $\omega \in \R$;
\item $\beta$ is bounded: $0 < \frac{1}{1 + |\omega |} \leq \beta(\omega) \leq 1$ for all $\omega \in \R$.
\end{enumerate}

Our result is

\begin{theorem}\label{T:2.1}
If $\psi \in L^2(\R) \setminus \{ 0 \}$ is such that $\hat{\psi}$ is continuous and 
\[
|\hat{\psi}(\xi)| \leq C (1+|\xi|)^{-r}
\]
for all $\xi \in \R$, with 
\[
r > \max\{ 1, \frac{\alpha}{2(1-\alpha)} \},
\]
then $\psi$ is admissible.
\end{theorem}

\textbf{Remark:} Note that, for $\alpha \to 1$, the exponent $r$ becomes larger and larger: $r \to \infty$. That means that the closer $\alpha$ is to $1$, the stronger decay of the Fourier transform we need to achieve admissibility. 

\begin{proof}
We have to show that there exist positive constants $A, B > 0$ such that
\[ \label{eq_star}
0 < A \leq m(\xi) \leq B < \infty \tag{$\ast$}
\]
for almost all $\xi \in \R$, where
\[
m(\xi) = \int_{\R} |\hat{\psi}(\beta(\omega)(\xi - \omega))|^2 \beta(\omega)\,d\omega.
\]
For simplicity of notation, set
\[
r_{\xi}(\omega) = \beta(\omega)(\xi - \omega) = \frac{\xi - \omega}{(1 + |\omega |)^{\alpha}},
\]
that means
\[
m(\xi) = \int_{\R} | \hat{\psi}(r_{\xi}(\omega)) |^2 \beta(\omega)\,d\omega.
\]
First consider $\alpha = 0$. In this case, $\beta(\omega) \equiv 1$, thus
\[
m(\xi) = \int_{\R} |\hat{\psi}(\xi - \omega)|^2\,d\omega = \| \hat{\psi} \|^2 = \| \psi \|^2
\]
independent of $\xi$, so \eqref{eq_star} is satisfied with $A = B = \| \psi \|^2$. We do not even need the assumptions of Theorem \ref{T:2.1}  here; in fact \emph{every} $\psi \in L^2(\R)$ is admissible in the case $\alpha = 0$.

Now assume $0 < \alpha < 1$. We will first prove three lemmata that yield some simplifications.

In the first lemma (and only there), we will temporarily write $m_{\psi}(\xi)$ to denote the dependence of $m(\xi)$ on the function $\psi$.

\begin{lemma}\label{L:2.2}
We have
\[
m_{\psi}(-\xi) = m_{\overline{\psi}}(\xi),
\]
(where $\overline{\psi}(x) = \overline{\psi(x)}$ denotes the complex conjugate function to $\psi$).
\end{lemma}
\begin{proof}
We have
\[
m_{\psi}(-\xi) = \int_{\R} | \hat{\psi}(r_{-\xi}(\omega)) |^2 \beta(\omega)\,d\omega.
\]
Now,
\[
r_{-\xi}(\omega) = \frac{-\xi - \omega}{(1 + |\omega |)^{\alpha}} = - \frac{\xi - (-\omega)}{(1 + | -\omega |)^{\alpha}} = - r_{\xi}(-\omega),
\]
so
\[
m_{\psi}(-\xi) = \int_{\R} | \hat{\psi}(-r_{\xi}(-\omega)) |^2 \beta(-\omega)\,d\omega = \int_{\R} | \hat{\psi}(-r_{\xi}(\omega)) |^2 \beta(\omega)\,d\omega.
\]
With $\overline{\hat{\psi}(-\eta)} = \hat{\overline{\psi}}(\eta)$, we conclude
\[
m_{\psi}(-\xi) = \int_{\R} | \hat{\overline{\psi}}(r_{\xi}(\omega)) |^2 \beta(\omega)\,d\omega = m_{\overline{\psi}}(\xi).
\]
\end{proof}
It is clear that if $\psi \in L^2(\R)$ satisfies the assumptions of Theorem \ref{T:2.1}, then $\overline{\psi}$ does so, as well. It suffices thus to show \eqref{eq_star} only for $\xi \geq 0$.

\begin{lemma}\label{L:2.3}
The function $m$ is strictly positive, i.e. $m(\xi) > 0$ for all $\xi \in \R$.
\end{lemma}
\begin{proof}
It is clear that $m(\xi) \geq 0$ for all $\xi$.
Suppose there is a $\xi \in \R$ such that $m(\xi) = 0$. Since the map $\omega \mapsto |\hat{\psi}(r_{\xi}(\omega))|^2 \beta(\omega)$ is continuous, $\int_{\R} |\hat{\psi}(r_{\xi}(\omega))|^2 \beta(\omega)\, d\omega = 0$ implies $|\hat{\psi}(r_{\xi}(\omega))|^2 \beta(\omega) = 0$ for all $\omega \in \R$, which gives
$|\hat{\psi}(r_{\xi}(\omega))| = 0$ for all $\omega \in \R$. Now observe that, for fixed $\xi$,
\[
r_{\xi}(\omega) = \frac{\xi - \omega}{(1 + |\omega |)^{\alpha}} \longrightarrow \begin{cases} -\infty & \quad \mbox{ for }\omega \to + \infty, \\
+ \infty & \quad \mbox{ for } \omega \to - \infty;
\end{cases}
\]
thus the range of $r_{\xi}$ is all of $\R$ and so $\hat{\psi}(\eta) = 0$ for all $\eta \in \R$. But this is equivalent to $\psi = 0$, a contradiction.
\end{proof}

\begin{lemma}\label{L:2.4}
The function $m(\xi)$ is continuous in $\xi$.
\end{lemma}
\begin{proof}
Let $\xi \in \R$ be fixed. Let $\xi^{\prime} \in \R$ with $|\xi^{\prime} - \xi | \leq \frac{1}{2}$. Then
\[
| m(\xi) - m(\xi^{\prime}) | \leq \int_{\R} \underbrace{\left| |\hat{\psi}(r_{\xi}(\omega))|^2 - | \hat{\psi}(r_{\xi^{\prime}}(\omega)) |^2 \right|\,\beta(\omega)}_{=: I(\omega)}\,d\omega.
\]
We want to use the Dominated Convergence Theorem on this last integral. To this end, consider the integrand $I(\omega)$. If $\psi$ satisfies $|\hat{\psi}(\xi)| \leq C (1+|\xi|)^{-r}$, we can estimate
\begin{align*}
I(\omega) & \leq \beta(\omega) \left( C^{2}(1 + | r_{\xi}(\omega) |)^{-2r} + C^{2}(1 + | r_{\xi^{\prime}}(\omega) |)^{-2r} \right) \\
& = C^{2} \beta(\omega) \left( (1 + | r_{\xi}(\omega) |)^{-2r} + (1 + | r_{\xi^{\prime}}(\omega) |)^{-2r} \right).
\end{align*}
Observe that if $|\xi^{\prime} - \xi | \leq \frac{1}{2}$ then
\begin{align*}
| r_{\xi}(\omega) - r_{\xi^{\prime}}(\omega) | & = |\beta(\omega)(\xi - \omega) - \beta(\omega)(\xi^{\prime} - \omega) | \\
& = \beta(\omega) | \xi - \xi^{\prime} | \leq \frac{1}{2}
\end{align*}
for all $\omega \in \R$. Hence
\[
1 + | r_{\xi^{\prime}}(\omega) | \geq \frac{1}{2} + | r_{\xi}(\omega) |
\]
for all $\omega\in\R$; since, trivially, also
\[ 1 + | r_{\xi}(\omega) | \geq \frac{1}{2} + | r_{\xi}(\omega) |,
\]
we get
\begin{align*}
I(\omega) & \leq C^{2} \beta(\omega) \cdot 2(\frac{1}{2} + | r_{\xi}(\omega) |)^{-2r} \\
& = \tilde{C} \beta(\omega) ( \frac{1}{2} + \beta(\omega) | \xi - \omega | )^{-2r},
\end{align*}
which is independent of $\xi^{\prime}$. But this last expression is integrable, since, for large $|\omega|$, it behaves asymptotically like
\[
\sim |\omega |^{-\alpha} (| \omega |^{-\alpha} |\omega |)^{-2r} = \frac{1}{| \omega |^{\alpha + 2r(1-\alpha)}},
\]
and with $r > 1$ the exponent satisfies
\[
\alpha + 2r(1-\alpha) > \alpha + 2(1 - \alpha) = 2 - \alpha > 1.
\]
It is further clear that for $\xi^{\prime} \to \xi$, we have $r_{\xi^{\prime}}(\omega) \to r_{\xi}(\omega)$ and, since $\hat{\psi}$ is continuous, $\hat{\psi}(r_{\xi^{\prime}}(\omega)) \to \hat{\psi}(r_{\xi}(\omega))$, pointwise for all $\omega \in \R$. Thus the integrand satisfies
$I(\omega) \to 0$ for $\xi^{\prime} \to \xi$, pointwise for all $\omega \in \R$, so Dominated Convergence finally yields
\[
m(\xi^{\prime}) - m(\xi) \to 0
\]
for $\xi^{\prime} \to \xi$, i.e. $m$ is continuous.
\end{proof}

The last two lemmata show that it suffices to prove \eqref{eq_star} only for sufficiently large $\xi$, say $\xi \geq K$; in this case, on the compact interval $[0, K]$, $m$ satisfies \eqref{eq_star} as well, since it is continuous and strictly positive there, so \eqref{eq_star} will hold for all $\xi \geq 0$. We will in fact be able to show an even stronger statement: we will prove that $\lim_{\xi \to \infty} m(\xi) = L > 0$ exists and is positive. From this, the above follows.

Without loss of generality, assume from now on that $\xi > \frac{2}{\alpha}$. Then, obviously, $\alpha \xi > 2$ and $\xi > 2$.

Before we proceed any further, we will discuss the function $r_{\xi}(\omega)$ in more detail.

\begin{lemma}\label{L:2.5}
The derivative of $r_{\xi}(\omega)$ for $\omega \not= 0$ is given by
\[
r^{\prime}_{\xi}(\omega) =  \frac{d}{d\omega}r_{\xi}(\omega) = -\beta(\omega) \left( 1 + \mbox{sgn}(\omega)\cdot \alpha \frac{\xi - \omega}{1 + |\omega |} \right).
\]
\end{lemma}
\begin{proof}
For $\omega > 0$, we have $r_{\xi}(\omega) = \frac{\xi - \omega}{(1 + \omega)^{\alpha}}$, which differentiates to
\begin{align*}
r^{\prime}_{\xi}(\omega) & = \frac{-(1 + \omega)^{\alpha} - (\xi - \omega) \alpha (1 + \omega)^{\alpha - 1}}{(1 + \omega)^{2\alpha}} \\
& = - \frac{1}{(1 + \omega)^{\alpha}} \left( 1 + \alpha \frac{\xi - \omega}{1 + \omega } \right).
\end{align*}
For $\omega < 0$, we have $r_{\xi}(\omega) = \frac{\xi - \omega}{(1 - \omega)^{\alpha}}$, which differentiates to
\begin{align*}
r^{\prime}_{\xi}(\omega) & = \frac{-(1 - \omega)^{\alpha} + (\xi - \omega) \alpha (1 - \omega)^{\alpha - 1}}{(1 - \omega)^{2\alpha}} \\
& = - \frac{1}{(1 - \omega)^{\alpha}} \left( 1 - \alpha \frac{\xi - \omega}{1 - \omega } \right).
\end{align*}
In summary,
\[
r^{\prime}_{\xi}(\omega) = -\frac{1}{(1 + |\omega |)^{\alpha}}\left( 1 + \mbox{sgn}(\omega)\cdot \alpha \frac{\xi - \omega}{1 + |\omega |} \right),
\]
as claimed.
\end{proof}

\begin{lemma}\label{L:2.6}
Set $\omega^{\ast}_{\xi} := \frac{1 - \alpha \xi}{1 - \alpha} < 0$. Then $r_{\xi}(\omega)$ is strictly decreasing on $]-\infty, \omega^{\ast}_{\xi}[$, strictly increasing on $]\omega^{\ast}_{\xi}, 0[$, and strictly decreasing on $]0, +\infty[$. At the local minimum $\omega^{\ast}_{\xi}$, we have $r_{\xi}(\omega^{\ast}_{\xi}) = \frac{1}{\alpha^{\alpha}} \left( \frac{\xi - 1}{1 - \alpha} \right)^{1 - \alpha} > 0$. At the local maximum 0, we have $r_{\xi}(0) = \xi > 0$.
\end{lemma}
\begin{proof}
Let $\omega > 0$. Then 
\begin{align*}
r^{\prime}_{\xi}(\omega) & = -\beta(\omega) \left( 1 + \alpha \frac{\xi - \omega}{1 + \omega} \right) \\
& = - \frac{\beta(\omega)}{1 + \omega}\left( 1 + \omega(1 - \alpha) + \alpha \xi \right) < 0,
\end{align*}
so $r_{\xi}(\omega)$ is strictly decreasing for $\omega > 0$.\\
For $\omega < 0$, note that $r^{\prime}_{\xi}(\omega) = -\beta(\omega) \left(1 - \alpha \frac{\xi - \omega}{1 - \omega}  \right) = 0$ if and only if $\left(1 - \alpha \frac{\xi - \omega}{1 - \omega}  \right) = 0$ if and only if $1 - \omega(1-\alpha) - \alpha \xi = 0$ if and only if $\omega = \frac{1 - \alpha \xi}{1 - \alpha} = \omega^{\ast}_{\xi}$. Since $\alpha \xi > 1$, $\omega^{\ast}_{\xi} < 0$. Since $r^{\prime}_{\xi}(\omega)$ is continuous on $]-\infty, 0[$, this is the only place where it can change its sign. If $\omega \to - \infty$, then $r^{\prime}_{\xi}(\omega) \uparrow 0$ and $r^{\prime}_{\xi}(\omega) < 0$, so $r^{\prime}_{\xi}(\omega) < 0$ for $\omega < \omega^{\ast}_{\xi}$, and $r_{\xi}(\omega)$ is decreasing. If $\omega \to 0$, then $r^{\prime}_{\xi}(\omega) \to \alpha \xi - 1 > 0$, so $r^{\prime}_{\xi}(\omega) > 0$ for $0 > \omega > \omega^{\ast}_{\xi}$, and $r_{\xi}(\omega)$ is increasing. The values at the local minimum $\omega_{\xi}^{\ast}$ and the local maximum $0$ follow by a simple computation.
\end{proof}

Our final lemma will help to compute and estimate several integrals in the following.

\begin{lemma}\label{L:2.7}
Let $I \subseteq \R$ be an interval such that $r_{\xi}$ is monotonous on $I$. Then
\[
\int_I | \hat{\psi}(r_{\xi}(\omega)) |^2 \beta(\omega)\,d\omega = \int_{r_{\xi}(I)} | \hat{\psi}(z) |^2 \frac{1}{| h_{\xi}(r_{\xi}^{-1}(z) |} \,dz,
\]
with
\[
h_{\xi}(\omega) = 1 + \mbox{sgn}(\omega)\cdot \alpha \frac{\xi - \omega}{1 + |\omega |}, \quad \quad \omega \in I.
\]
\end{lemma}
\begin{proof}
We necessarily have $0 \not\in \mbox{int}(I)$, the interior of $I$, since, by Lemma \ref{L:2.6}, $r_{\xi}$ has a local maximum in $0$, so $r_{\xi}$ is not monotonous if $0 \in \mbox{int}(I)$.

Observe that, by Lemma \ref{L:2.5},
\[
r_{\xi}^{\prime}(\omega) = -\beta(\omega) h_{\xi}(\omega);
\]
the statement follows from this by the substitution $z = r_{\xi}(\omega)$, $dz = -\beta(\omega) h_{\xi}(\omega)\,d\omega$. 
\end{proof}

We are now ready to finish the proof of Theorem \ref{T:2.1}.

We split the integral defining $m(\xi)$ into four parts,
\[
m(\xi) = \int_{\R} |\hat{\psi}(r_{\xi}(\omega)) |^2 \beta(\omega)\, d\omega = \int_{I_1}\ldots + \int_{I_2}\ldots + \int_{I_3}\ldots + \int_{I_4}\ldots,
\]
and treat each part separately. The four intervals are
\begin{align*}
I_1 & = (-\infty, \omega_{\xi}^{\ast} - \frac{\alpha \xi^{\alpha}}{2(1 - \alpha)}], \\
I_2 & = [\omega_{\xi}^{\ast} - \frac{\alpha \xi^{\alpha}}{2(1 - \alpha)}, \omega_{\xi}^{\ast} + \frac{\alpha \xi^{\alpha}}{2(1 - \alpha)}],\\
I_3 & = [\omega_{\xi}^{\ast} + \frac{\alpha \xi^{\alpha}}{2(1 - \alpha)}, 0], \mbox{ and } \\
I_4 & = [0, \infty).
\end{align*}

Observe that
\begin{align*}
\omega_{\xi}^{\ast} + \frac{\alpha \xi^{\alpha}}{2(1 - \alpha)}  = \frac{1 - \alpha \xi}{1 - \alpha} + \frac{\frac{\alpha}{2} \xi^{\alpha}}{1 - \alpha}  < \frac{1 - \alpha \xi + \frac{\alpha}{2}\xi}{1 - \alpha}  = \frac{1 - \frac{\alpha}{2}\xi}{1 - \alpha} < 0,
\end{align*}
since $\xi > \frac{2}{\alpha}$. Thus $I_2 \subset (-\infty, 0]$, and $I_3$ is well defined.

Note that $r_{\xi}$ is monotonous on $I_1$, $I_3$ and $I_4$.

\begin{itemize}

\item \textbf{$I_2 = [\omega_{\xi}^{\ast} - \frac{\alpha \xi^{\alpha}}{2(1 - \alpha)}, \omega_{\xi}^{\ast} + \frac{\alpha \xi^{\alpha}}{2(1 - \alpha)}]$:} On $I_2$, $r_{\xi}$ has a local minimum at $\omega_{\xi}^{\ast}$. Thus
\begin{align*}
\int_{I_2} & |\hat{\psi}(r_{\xi}(\omega)) |^2 \beta(\omega)\, d\omega \leq \int_{I_2} C^2 (1 + | r_{\xi}(\omega) |)^{-2r} \, d\omega \\
& \leq \int_{I_2} C^2 (1 + | r_{\xi}(\omega_{\xi}^{\ast}) |)^{-2r}\, d\omega \\
& = \frac{\alpha}{1 - \alpha} C^2 \xi^{\alpha} (1 + | r_{\xi}(\omega_{\xi}^{\ast}) |)^{-2r}.
\end{align*}
By Lemma \ref{L:2.6}, we asymptotically have $| r_{\xi}(\omega_{\xi}^{\ast}) | \sim | \xi |^{1 - \alpha}$, so
\[
\int_{I_2} |\hat{\psi}(r_{\xi}(\omega)) |^2 \beta(\omega)\, d\omega \sim |\xi |^{\alpha} (| \xi |^{1 - \alpha})^{-2r} = |\xi |^{\alpha - 2r(1 - \alpha)};
\]
since $r > \frac{\alpha}{2(1-\alpha)}$, we have $\alpha - 2r(1 - \alpha) < 0$, thus
\[
\int_{I_2} |\hat{\psi}(r_{\xi}(\omega)) |^2 \beta(\omega)\, d\omega \to 0
\]
for $\xi \to \infty$.

\item \textbf{$I_1 = (-\infty, \omega_{\xi}^{\ast} - \frac{\alpha \xi^{\alpha}}{2(1 - \alpha)}]$:} We use Lemma \ref{L:2.7} to write
\[
\int_{I_1} | \hat{\psi}(r_{\xi}(\omega)) |^2 \beta(\omega) \, d\omega = \int_{r_{\xi}(I_1)} | \hat{\psi}(z) |^2 \frac{1}{| h_{\xi}(r_{\xi}^{-1}(z) |} \,dz,
\]
where in this case
\[
h_{\xi}(\omega) = 1 - \alpha \frac{\xi - \omega}{1 - \omega }, \quad \quad \omega \in I_1.
\]
Since $h_{\xi}(\omega) = - \frac{r_{\xi}^{\prime}(\omega)}{\beta(\omega)}$, we have $h_{\xi}(\omega) > 0$ on the interval \\ 
$(-\infty, \omega_{\xi}^{\ast}[$, by Lemma \ref{L:2.6}. Furthermore,
\[
h_{\xi}^{\prime}(\omega) = - \alpha \frac{\xi - 1}{(1 - \omega)^2} < 0
\]
on $(-\infty, \omega_{\xi}^{\ast}[$, so $h_{\xi}$ is a monotone decreasing function on $I_1$, and hence $h_{\xi}$ assumes its infimum at the rightmost point of $I_1$, i.e. at $\omega_{\xi}^{\ast} - \frac{\alpha \xi^{\alpha}}{2(1 - \alpha)}$. The infimum is given by
\[
\inf_{\omega \in I_1} | h_{\xi}(\omega)| = | h_{\xi}(\omega_{\xi}^{\ast} - \frac{\alpha \xi^{\alpha}}{2(1 - \alpha)}) | = \ldots = \frac{\frac{1-\alpha}{2}\xi^{\alpha}}{\xi - 1 + \frac{1}{2}\xi^{\alpha}},
\]
which behaves asymptotically like $| \xi|^{\alpha - 1}$. We conclude
\[
\frac{1}{| h_{\xi}(\omega) |} \leq \frac{1}{\inf_{\omega \in I_1} | h_{\xi}(\omega) |} \sim | \xi|^{1 - \alpha},
\]
for $\omega \in I_1$. For the transformed interval, we find
$r_{\xi}(I_1)  = [z_1(\xi), + \infty)$ with
\begin{align*}
z_1(\xi) & = r_{\xi}(\omega_{\xi}^{\ast} - \frac{\alpha \xi^{\alpha}}{2(1 - \alpha)}) \\
& = r_{\xi}(\frac{1 - \alpha \xi - \frac{\alpha}{2}\xi^{\alpha}}{1 - \alpha}) \\
& = \ldots = \frac{1}{(1-\alpha)^{1 - \alpha}\alpha^{\alpha}} \frac{\xi - 1 + \frac{\alpha}{2}\xi^{\alpha}}{(\xi - 1 + \frac{1}{2}\xi^{\alpha})^{\alpha}},
\end{align*}
which also behaves like $| \xi|^{1 - \alpha}$. Putting it all together, we find
\begin{align*}
\int_{r_{\xi}(I_1)} | \hat{\psi}(z) |^2 & \frac{1}{| h_{\xi}(r_{\xi}^{-1}(z) |} \,dz \\
& \leq \frac{1}{\inf_{\omega \in I_1} | h_{\xi}(\omega) |} \int_{r_{\xi}(I_1)} | \hat{\psi}(z) |^2\,dz \\
& \leq \frac{1}{\inf_{\omega \in I_1} | h_{\xi}(\omega)| } \int_{z_1(\xi)}^{\infty} C^2(1 + |z|)^{-2r}\,dz \\
& = \frac{C^2}{(2r - 1)\cdot \inf_{\omega \in I_1} | h_{\xi}(\omega)| } (1 + z_1(\xi))^{-2r + 1},
\end{align*}
that is asymptotically equivalent to $|\xi|^{1-\alpha} |\xi |^{(1-\alpha)(-2r + 1)} = \\ |\xi |^{2(1-\alpha)(1-r)}$. Since $2(1-\alpha)(1 - r) < 0$, we finally conclude
\[
\int_{I_1} | \hat{\psi}(r_{\xi}(\omega)) |^2 \beta(\omega) \, d\omega = \int_{r_{\xi}(I_1)} | \hat{\psi}(z) |^2  \frac{1}{| h_{\xi}(r_{\xi}^{-1}(z) |} \,dz \to 0
\]
for $\xi \to \infty$.

\item \textbf{$I_3 = [\omega_{\xi}^{\ast} + \frac{\alpha \xi^{\alpha}}{2(1 - \alpha)}, 0]$:} This is very similar to the previous case $I_1$. We have
\[
\int_{I_3} | \hat{\psi}(r_{\xi}(\omega)) |^2 \beta(\omega) \, d\omega = \int_{z_2(\xi)}^{\xi} | \hat{\psi}(z) |^2 \frac{1}{| h_{\xi}(r_{\xi}^{-1}(z) |} \,dz
\]
with $h_{\xi}$ as above, and
\[
z_2(\xi) = \frac{1}{(1-\alpha)^{1 - \alpha}\alpha^{\alpha}} \frac{\xi - 1 - \frac{\alpha}{2}\xi^{\alpha}}{(\xi - 1 - \frac{1}{2}\xi^{\alpha})^{\alpha}}.
\]
On $I_3$, $h_{\xi}(\omega) < 0$, $h^{\prime}_{\xi}(\omega) < 0$, so
\[
\inf_{\omega \in I_3} | h_{\xi}(\omega)| = | h_{\xi}(\omega_{\xi}^{\ast} + \frac{\alpha \xi^{\alpha}}{2(1 - \alpha)}) | = \ldots = \frac{\frac{1-\alpha}{2}\xi^{\alpha}}{\xi - 1 - \frac{1}{2}\xi^{\alpha}} \sim |\xi |^{\alpha - 1}.
\]
This yields
\begin{align*}
\int_{I_3} | \hat{\psi}(r_{\xi}(\omega)) & |^2 \beta(\omega) \, d\omega \\
& \leq \frac{1}{\inf_{\omega \in I_3} | h_{\xi}(\omega)| } \int_{z_2(\xi)}^{\xi} C^2(1 + |z|)^{-2r}\,dz \\
& \sim |\xi |^{2(1-\alpha)(1-r)},
\end{align*}
and this goes to $0$ for $\xi \to \infty$ because $2(1-\alpha)(1 - r) < 0$.

\item \textbf{$I_4 = [0, \infty)$:} We consider
\[
\int_{I_4} | \hat{\psi}(r_{\xi}(\omega)) |^2 \beta(\omega)\,d\omega = \int_{r_{\xi}(I_4)} | \hat{\psi}(z) |^2 \frac{1}{| h_{\xi}(r_{\xi}^{-1}(z) |} \,dz,
\]
now with
\[
h_{\xi}(\omega) = 1 + \alpha \frac{\xi - \omega}{1 + \omega }, \quad \quad \omega \in I_4.
\]
Since $r_{\xi}(0) = \xi$ and $\lim_{\omega \to \infty} r_{\xi}(\omega) = -\infty$, we have $r_{\xi}(I_4) = (-\infty, \xi]$. Let $\varepsilon > 0$ be given. Choose $A > 0$ such that
\[
\int_{\R \setminus [-A, A]} | \hat{\psi}(z) |^2\,dz \leq \varepsilon,
\]
that means
\[
\left| \int_{[-A, A]} | \hat{\psi}(z) |^2\,dz - \| \psi \|^2 \right| \leq \varepsilon.
\]
Now assume $\xi > A$. Then
\[
\int_{-\infty}^{\xi} | \hat{\psi}(z) |^2 \frac{1}{| h_{\xi}(r_{\xi}^{-1}(z) |} \,dz  = \int_{[-A, A]}\ldots + \int_{(-\infty, \xi] \setminus [-A, A]}\ldots.
\]
The second integral can be estimated as follows: first observe that
\[
h_{\xi}(\omega) = \frac{1 + \alpha \xi + \omega(1 - \alpha)}{1 + \omega} > 0
\]
on $I_4$. Its derivative is
\[
h_{\xi}^{\prime}(\omega) = - \alpha \frac{\xi + 1}{(1 + \omega)^2} < 0,
\]
so $\inf_{\omega \in I_4} |h_{\xi}(\omega) | = \lim_{\omega \to \infty} | h_{\xi}(\omega) | = 1 - \alpha$. Hence
\[
\int_{(-\infty, \xi] \setminus [-A, A]}\ldots \leq \frac{1}{1 - \alpha} \int_{(-\infty, \xi] \setminus [-A, A]} | \hat{\psi}(z) |^2\,dz \leq \frac{\varepsilon}{1 - \alpha}.
\]
For the first integral, we use the following statement that is taken from \cite{DahForRau+2008} (see Lemma 5.1 and the proof of Theorem 5.2 there):

\begin{lemma}\label{L:2.8}
For every fixed $A > 0$,
\[
\lim_{\xi \to \infty} h_{\xi}(r_{\xi}^{-1}(z)) = 1
\]
uniformly on $[-A, A]$.
\end{lemma}

This yields
\[
\int_{[-A, A]}\ldots \to \int_{[-A, A]} |\hat{\psi}(z)|^2\,dz
\]
for $\xi \to \infty$. Thus
\[
\left| \int_{I_4} | \hat{\psi}(r_{\xi}(\omega)) |^2 \beta(\omega)\,d\omega - \| \psi \|^2 \right| \leq \varepsilon + \frac{\varepsilon}{1 - \alpha}.
\]
Since $\varepsilon$ was arbitrary, we conclude
\[
\int_{I_4} | \hat{\psi}(r_{\xi}(\omega)) |^2 \beta(\omega)\,d\omega \to \| \psi \|^2
\]
for $\xi \to \infty$.

\end{itemize}

We have thus shown
\[
m(\xi) = \int_{I_1}\ldots + \int_{I_2}\ldots + \int_{I_3}\ldots + \int_{I_4}\ldots \longrightarrow \|\psi\|^2
\]
for $\xi \to \infty$, which finally concludes the proof of Theorem \ref{T:2.1}.
\end{proof}



The assumptions of Theorem \ref{T:2.1} are in particular satisfied for $\psi \in \mathcal{S}(\R) \subseteq L^2(\R)$, the Schwartz class of infinitely differentiable rapidly decaying functions: their Fourier transforms are again of the same class, thus decay faster than any given polynomial. Since there exist Schwartz functions with compact support, this proves the existence of compactly supported admissible functions for $\alpha$-modulation spaces.


\section*{Acknowledgement}

The authors wish to thank S. Dahlke for his invaluable help and numerous improvements.

This work was funded by project BIOTOP, FWF Project Number I-1018-N25.


\bibliographystyle{plain}
\bibliography{biotop_coorbit}

\end{document}